\documentclass[11pt]{article}

\usepackage[margin=1.1in]{geometry}
\usepackage{amsmath,amssymb,amsthm,mathtools}
\usepackage{bm}
\usepackage{enumitem}
\usepackage{microtype}
\usepackage[hidelinks]{hyperref}

\newcommand{\R}{\mathbb{R}}
\newcommand{\Z}{\mathbb{Z}}
\newcommand{\Sph}{\mathbb{S}}
\newcommand{\T}{\mathbb{T}}
\newcommand{\vol}{\operatorname{vol}}
\newcommand{\Cov}{\operatorname{Cov}}
\newcommand{\Var}{\operatorname{Var}}
\newcommand{\Per}{\operatorname{Per}}

\newcommand{\dist}{\operatorname{dist}}
\newcommand{\tr}{\operatorname{tr}}
\newcommand{\sg}{\lambda_{\mathrm{SG}}}
\newcommand{\la}{\lambda_1}
\newcommand{\norm}[1]{\left\lVert #1\right\rVert}
\newcommand{\ip}[2]{\left\langle #1,#2\right\rangle}
\newcommand{\1}{\mathbf{1}}

\newcommand{\scalar}[1]{\left \langle #1 \right \rangle}

\newcommand{\I}{\mathcal{I}}

\theoremstyle{plain}
\newtheorem{theorem}{Theorem}[section]

\newtheorem{lemma}[theorem]{Lemma}
\newtheorem{corollary}[theorem]{Corollary}
\newtheorem{remark}[theorem]{Remark}
\newtheorem*{fact*}{Fact}

\theoremstyle{definition}
\newtheorem{example}[theorem]{Example}

\begin{document}

\renewcommand*{\thefootnote}{\fnsymbol{footnote}}

\author{Emanuel Milman\thanks{Technion Israel Institute of Technology, Department of Mathematics, Haifa 32000, Israel. Email: emilman@tx.technion.ac.il.}
}

\begingroup    \renewcommand{\thefootnote}{}            \footnotetext{The research leading to these results is part of a project that has received funding from the European Research Council (ERC) under the European Union's Horizon 2020 research and innovation programme (grant agreement No 101001677).}
\endgroup

\title{Spectral and Isoperimetric Bounds on Flat Tori}
\date{}

\maketitle

\begin{abstract}
We record several elementary relations between spectral and isoperimetric parameters of a flat torus $\T_\Lambda = \R^n/\Lambda$ and the covariance structure of a fundamental domain $K$ for the lattice $\Lambda \subset \R^n$.  For every measurable fundamental domain $K$ and nonzero vector $\xi$ in the dual lattice $\Lambda^*$, we observe the sharp directional variance estimate
\[
    \ip{\Cov_K\xi}{\xi}\ge \frac1{12}.
\]
This yields a sharp lower bound on the torus spectral gap $\lambda_{SG}(\T_\Lambda)$ (equivalently, the length of the shortest non-zero dual vector $\lambda_1(\Lambda^*)$) in terms of the maximal covariance of $K$:
\begin{equation} \label{eq:abstract}
 \lambda_{SG}(\T_\Lambda)  = 4 \pi^2 \lambda_1(\Lambda^*)^2 \geq \frac{\pi^2}{3 \norm{\Cov_K}_{op}} . 
\end{equation}
Analogous sharp results are obtained for the isoperimetric profile and the Cheeger constant $D_{Che}(\T_\Lambda)$ using an old argument of Hadwiger.

In particular, when the Voronoi cell $K_{\Lambda}$ of a lattice with $\det \Lambda = 1$ is isotropic, the recent resolution of the Slicing Problem by Klartag and Lehec implies that $D_{Che}(\T_\Lambda),\lambda_{SG}(\T_\Lambda), \lambda_1(\Lambda^*) \geq c > 0$, where $c > 0$ is a universal constant independent of dimension $n$; this may be thought of as a positive resolution of the Kannan--Lov\'asz--Simonovits conjecture for all flat tori. 

While there are lattices $\Lambda$ and corresponding Voronoi cells $K = K_{\Lambda}$ for which no dimension-independent converse inequality to (\ref{eq:abstract}) can hold, we show that under a certain sectional tiling hypothesis, (\ref{eq:abstract}) is in fact an equivalence (up to numeric constants). \end{abstract}

\section{Introduction}

Let $\Lambda\subset\R^n$ be a full-rank lattice and let
\[
    \T_\Lambda:=\R^n/\Lambda
\]
be the associated flat torus. Our goal in this note is to establish quantitative relations between spectral and isoperimetric parameters of $\T_\Lambda$ and the covariance structure of a fundamental domain $K$ of $\Lambda$, most notably its central (convex) Voronoi cell $K_{\Lambda}$. 

The dual lattice is defined as
\[
    \Lambda^*:=\{\xi\in\R^n:\ip{\xi}{\lambda}\in\Z\ \text{for all }\lambda\in\Lambda\},
\]
and the shortest nonzero dual vector is denoted by
\[
    \la(\Lambda^*):=\min\{ |\xi|:0\neq \xi\in\Lambda^*\} . 
\]
A complete system of eigenfunctions for the nonnegative Laplacian $-\Delta$ on $\T_\Lambda$ is given by
\[
    x\longmapsto e^{2\pi i\ip{\xi}{x}},\qquad \xi\in\Lambda^*,
\]
with corresponding eigenvalues $4\pi^2|\xi|^2$ \cite{Fuglede-SpectralSets}.  Consequently, the first nonzero eigenvalue (``the spectral gap") is given by
\[
\lambda_{SG}(\T_\Lambda) = 4 \pi^2 \lambda_1^2(\Lambda^*) . 
\]

A measurable set $K\subset\R^n$ is called a fundamental domain for $\Lambda$ if its translates by $\Lambda$ tile $\R^n$ up to null sets.  We write $|K|$ for its Lebesgue measure (so $|K| =\det\Lambda$). If $K$ has finite second moment, we denote by $b_K$ its barycenter and by $\Cov_K$ its covariance matrix
\[
    \Cov_K:=\frac1{|K|}\int_K (x-b_K)\otimes(x-b_K)\,dx.
\]
If $X$ is uniformly distributed in $K$ and $\xi \in \R^n$, note that $\ip{\Cov_K \xi }{\xi} = \textrm{Var}(\scalar{X,\xi})$. 
The central Voronoi cell is the convex fundamental domain
\[
    K_\Lambda:=\big\{x\in\R^n:|x|\le |x-\lambda|\ \text{for every }\lambda\in\Lambda\big\}.
\]
Equivalently,
\[
    K_\Lambda=\bigcap_{\lambda\in\Lambda\setminus\{0\}}
       \left\{x:\ip{x}{\lambda}\le \frac{|\lambda|^2}{2}\right\}.
\]
It is centrally symmetric and hence has barycenter at the origin.

\subsection{Spectral bounds}

\begin{lemma}[Dual directions have uniformly large variance]\label{lem:variance}
Let $K$ be any measurable fundamental domain of $\Lambda$ with finite second moment.  Then for every $0\neq\xi\in\Lambda^*$,
\[
    \ip{\Cov_K\xi}{\xi}\ge \frac1{12}.
\]
In particular,
\begin{equation} \label{eq:intro-SG}
    \sg(\T_\Lambda)= 4 \pi^2 \la(\Lambda^*)^2
       \ge \frac{\pi^2}{3 \norm{\Cov_K}_{\mathrm{op}}}.
\end{equation}
\end{lemma}

The constant $1/12$ is immediately seen to be sharp by considering $\Lambda = \Z^n$, $K = [-1/2,1/2]^n$ and coordinate vectors $\xi = e_i$. 

\subsection{Isoperimetric bounds}

There is an analogous statement for the isoperimetric profile of $\T_\Lambda$. Let us normalize the Haar measure $\mu_{\Lambda}$ on $\T_\Lambda$ to be a probability measure.  The corresponding isoperimetric profile is then defined as
\[
    \I_{\T_\Lambda}(v):=\inf \big\{\Per_{\T_\Lambda}(A) : A \subset \T_{\Lambda} ~,~ \mu_\Lambda(A)=v\big\},\qquad 0\le v\le1,
\]
where the infimum ranges over all subsets $A$ of finite perimeter, and the perimeter in $\T_{\Lambda}$ is normalized by $\det\Lambda = \vol(\T_\Lambda)$ in the same way as volume. The corresponding Cheeger constant is defined as
\[
D_{Che}(\T_{\Lambda}) := \inf_{v \in (0,1)} \frac{\I_{\T_{\Lambda}}(v)}{\min(v,1-v)} . 
\]
Since $\T_{\Lambda}$ is flat (and in particular of non-negative Ricci curvature), it is known \cite{SternbergZumbrun1999,Bayle2004} that the isoperimetric profile is concave (and symmetric about $v=1/2$), and hence the infimum above is attained at $v=1/2$, so that $D_{Che}(\T_{\Lambda}) = 2 \I_{\T_{\Lambda}}(1/2)$. Furthermore, we have:
\begin{equation} \label{eq:intro-Buser}
\frac{1}{4} D_{Che}(\T_{\Lambda})^2 \leq  \lambda_{SG}(\T_\Lambda)  \leq \pi D_{Che}(\T_{\Lambda})^2  \; ;
\end{equation}
the first inequality above is due to Cheeger \cite{Cheeger1970}, and the second is De Ponti--Mondino's improvement \cite{DePontiMondino2021} of the numeric constant in Buser's inequality \cite{Buser1982} under non-negative Ricci curvature.

For a measurable fundamental domain $K$ with finite first moment, define
\[
   R_1(K):=\sup_{\theta\in\Sph^{n-1}}
       \frac1{|K|}\int_K |\ip{x}{\theta}|\,dx \;  ;
\]
this parameter can be interpreted as 
the circumradius of the $L^1$-centroid body associated to $K$.

\begin{theorem}[Hadwiger's isoperimetric bound]\label{thm:iso}
For every measurable fundamental domain $K$ of $\Lambda$ with finite first moment,
\[
   \I_{\T_\Lambda}(v)\ge \frac{2v(1-v)}{R_1(K)},
   \;\; \forall v \in [0,1]. 
\]
In particular, it holds that
\begin{equation} \label{eq:intro-Cheeger}
D_{Che}(\T_{\Lambda}) \geq \frac{1}{R_1(K)} . 
\end{equation}
\end{theorem}

Theorem \ref{thm:iso} is due to Hadwiger \cite{Hadwiger1972} (see also \cite{Ros2005}), who originally proved it for the lattice $\Lambda = \Z^n$; for completeness (and as the original text is in German), we recall the proof. Instead of using a Crofton-type formula for sets with polygonal boundary like Hadwiger, we use an alternative argument which directly applies to general sets of finite perimeter. When $\Lambda = \Z^n$, one may show that $R_1([-1/2,1/2]^n) = 1/4$, and so Theorem \ref{thm:iso} yields a sharp result for the Cheeger constant of the standard flat torus (and hence, by a standard reflection argument, also for the unit cube $[-1/2,1/2]^n$).

\medskip

Note that both Lemma~\ref{lem:variance} and Theorem~\ref{thm:iso} hold for \emph{any} fundamental domain, and that different domains will yield significantly different bounds. For example, for $\Lambda=\Z^2$ the parallelogram generated by $(1,0)$ and $(N,1)$ is again a fundamental domain, but its maximal covariance is of order $N^2$. Presumably it is always preferable to apply these estimates to the central Voronoi cell $K_{\Lambda}$. 

\medskip

Also note that whenever the barycenter of $K$ is at the origin, $R_1(K)^2 \leq \norm{\Cov_K}_{op}$ by Jensen's inequality, and so together with (\ref{eq:intro-Buser}), (\ref{eq:intro-Cheeger}) implies (\ref{eq:intro-SG}) up to the value of the numeric constants. When $K$ is in addition convex, we have $\norm{\Cov_K}_{op} \leq 2 R_1(K)^2$ \cite{MelbourneRoysdonTangTkocz2026}, and so (\ref{eq:intro-Cheeger}) and (\ref{eq:intro-SG}) are equivalent up to constants.

\subsection{The KLS conjecture on flat tori}

An isoperimetric conjecture of Kannan--Lov\'asz--Simonovits (KLS) \cite{KLS1995} asserts that for any convex body $K \subset \R^n$, one has:
\begin{equation} \label{eq:intro-KLS}
D_{Che}(K)^2 \geq \frac{c}{\norm{\Cov_K}_{op}} ,
\end{equation}
where $c > 0$ is a dimension-independent constant (and the Cheeger constant $D_{Che}(K)$ is defined analogously as for $\T_{\Lambda}$). Equivalently, by using (\ref{eq:intro-Buser}), the same inequality is conjectured for the Neumann spectral gap $\lambda_{SG}(K)$ (the first nonzero eigenvalue of the Neumann Laplacian on $K$), see e.g.~\cite{EMilman-RoleOfConvexity}. The conjecture easily reduces to the case when $K$ is isotropic, i.e.~its covariance is a multiple of the identity:
\[
 \Cov_K = L_K^2 |K|^{2/n} \mathrm{Id} ,
\]
where $L_K$ is called the isotropic constant of $K$. Thanks to the recent resolution of Bourgain's Slicing Problem by Klartag and Lehec~\cite{KlartagLehec2025}, it is known that $L_K \leq C$ for some dimension-independent constant $C > 0$ (the complementary bound $L_K \geq c > 0$ is elementary).

We see that Lemma~\ref{lem:variance} and Theorem~\ref{thm:iso}, together with the subsequent comments, confirm an analogue to the KLS conjecture for all flat tori (in fact, with respect to the covariance structure of any fundamental domain $K$). Moreover, when $K$ is convex and isotropic, the positive resolution of the Slicing Problem confirms that $\Cov_K$ is entirely controlled by the volume $|K|$. Consequently, one obtains the following appealing result:

\begin{corollary}[Isotropic Voronoi cells have uniformly large dual systoles, KLS conjecture on isotropic flat tori]\label{cor:isotropic}
Assume $\det\Lambda=1$ and that $K_\Lambda$ is isotropic, i.e.
\[
     \Cov_{K_\Lambda}=L_{K_\Lambda}^2\,\mathrm{Id}.
\]
Then
\[
     D_{Che}(\T_\Lambda) \geq c ~,~ \sg(\T_\Lambda) \geq c ~,~ \la(\Lambda^*) \ge c,
\]
where $c> 0$ is a universal numerical constant, independent of dimension $n$.
\end{corollary}

In fact, there is nothing special about the central Voronoi cell $K_\Lambda$, and any isotropic convex fundamental domain $K$ for $\Lambda$ would yield the same consequence, but we chose to formulate things for a concrete convex fundamental domain. 

\bigskip

We mention that thanks to recent progress, the KLS conjecture has almost been fully resolved: by combining Klartag's improved Lichnerowicz estimate \cite{Klartag2023} with the recent Chen--Klartag dimension-independent bound on an associated three-tensor \cite{ChenKlartag2026}, (\ref{eq:intro-KLS}) is known to hold with $c_n =c  \log(1+n)^{-1/2}$, as explicitly recorded in \cite{Letwin2026} (see also \cite{Bizeul2026}). In other words, for any convex $K \subset \R^n$, 
\begin{equation} \label{eq:intro-almost-KLS}
\norm{\Cov_K}_{op}^{-1} \geq \lambda_{SG}(K) \geq \frac{c}{\sqrt{\log(1+n)}} \norm{\Cov_K}_{op}^{-1} . 
\end{equation}

\subsection{Counterexamples to equivalence} \label{subsec:counter}

In view of Lemma~\ref{lem:variance}, it is natural to ask whether the Voronoi cell might satisfy a reverse bound
\begin{equation} \label{eq:intro-reverse}
    \norm{\Cov_{K_\Lambda}}_{\mathrm{op}}\,\la(\Lambda^*)^2\le C
\end{equation}
with some universal constant $C > 0$.  Unfortunately, this is in general false. To see this, recall the sharp estimate obtained by Autissier~\cite{Autissier2013} and Magazinov~\cite{Magazinov2020}
\begin{equation} \label{eq:Magazinov}
 \left ( n \norm{\Cov_{K_\Lambda}}_{op} \geq \; \right) \; \tr(\Cov_{K_\Lambda}) =   \frac1{|K_\Lambda|}\int_{K_\Lambda}|x|^2\,dx\ge \frac{R(\Lambda)^2}{3},
\end{equation}
where $R(\Lambda)=\max_{x\in K_\Lambda}|x|$ is the covering radius.  On the other hand, denoting
\[
    T_n:=\frac1n\sup_{\Lambda\subset\R^n}R(\Lambda)\la(\Lambda^*),
\]
it was shown by Banaszczyk \cite[bottom of p. 633]{Banaszczyk1993} that
\begin{equation} \label{eq:Tn}
T_n \ge \frac1{2\pi e}-o(1) . 
\end{equation}
An alternative argument yielding (\ref{eq:Tn}) was suggested in~\cite{AggarwalStephensDavidowitz2019}: if $r_n B_2^n$ is the volume $1$ centered Euclidean ball in $\R^n$, Siegel's mean-value theorem \cite{Siegel1945} implies that for every $\varepsilon > 0$, there is a lattice $\Lambda^*_\varepsilon$ in $\R^n$ with $\det \Lambda^*_\varepsilon = 1$ which does not contain any non-zero point of $(1-\varepsilon) r_n B^n$, so $\la(\Lambda^*_\varepsilon) \geq (1-\varepsilon) r_n$; on the other hand, $K_{\Lambda_\varepsilon} \subset R(\Lambda_\varepsilon) B_2^n$ and hence $R(\Lambda_\varepsilon) \geq |K_{\Lambda_\varepsilon}| |B_2^n|^{-1} = \det(\Lambda_\varepsilon) r_n  = r_n$. Therefore $T_n \geq \frac{1}{n} r_n^2$, and since $r_n = (\frac{1}{\sqrt{2\pi e}} - o(1)) \sqrt{n}$, this yields (\ref{eq:Tn}). 
Combining (\ref{eq:Magazinov}) and (\ref{eq:Tn}) gives the following.

\begin{corollary}[No dimension-free converse for Voronoi cells]\label{fact:counterexample}
There exist lattices $\Lambda_n\subset\R^n$ such that
\[
    \norm{\Cov_{K_{\Lambda_n}}}_{\mathrm{op}}\,\la(\Lambda_n^*)^2
       \ge \left(\frac1{12\pi^2e^2}-o(1)\right)n.
\]
In particular, no dimension-independent reverse inequality to Lemma~\ref{lem:variance} can hold for arbitrary lattice Voronoi cells.
\end{corollary}

One might still hope that (\ref{eq:intro-reverse}) should hold for lattices $\Lambda$ (say with $\det \Lambda = 1$) whose Voronoi cell $K_{\Lambda}$ is isotropic; by Corollary \ref{cor:isotropic}, this would imply that $\lambda_{SG}(\T_{\Lambda}) = 4 \pi^2 \lambda_1(\Lambda^*)^2$ remains bounded above and below for such lattices by dimension-independent constants. Unfortunately, this is also equally false -- see Example \ref{ex:counter} and Remark \ref{rem:counter}.

\subsection{Equivalence under sectional tiling}

However, a reverse inequality does in fact hold under a simple structural condition, involving a hyperplane tiling section in a direction of large variance:

\begin{theorem}[A sectional tiling criterion]\label{thm:section}
Let $K\subset\R^n$ be a convex fundamental domain of $\Lambda$ with barycenter at the origin. Suppose that there exists a linear hyperplane $H = \nu^{\perp} \subset\R^n$ ($\nu \in \Sph^{n-1}$) such that:
\begin{enumerate}[label=(\roman*)]
\item $K\cap H$ tiles $H$ by the translations $\Lambda\cap H$;
\item for some $\delta\in(0,1]$, $\scalar{\Cov_K \nu, \nu} \geq \delta \norm{\Cov_K}_{op}$. 
\end{enumerate}
Then
\[
   \frac1{12}
      \le \norm{\Cov_K}_{\mathrm{op}} \la(\Lambda^*)^2 = \frac{1}{4 \pi^2} \norm{\Cov_K}_{\mathrm{op}} \lambda_{SG}(\T_{\Lambda}) 
      \le \frac{1}{2\delta}.
\]
\end{theorem}

In view of (\ref{eq:intro-almost-KLS}), this implies under the above conditions:
\[
\frac{\pi^2}{3} \leq \frac{\lambda_{SG}(\T_\Lambda)}{\lambda_{SG}(K)} \leq \frac{C}{\delta} \sqrt{\log (1+n)} ,
\]
so the spectral-gaps of $\T_{\Lambda}$ and its fundamental domain $K$ are almost of the same order of magnitude in $n$. 
It is also possible to obtain a precise relation between these spectral-gaps 
if one can tile $\R^n$ by repeatedly reflecting $K$ across its facets (see e.g.~\cite{HoshikawaUrakawa2010}), but this is a much stronger condition which we do not pursue here. 

\medskip

\textbf{Acknowledgments.} I thank Barak Weiss and the anonymous referee for helpful references. 

\medskip

\textbf{AI Declaration:} All of the results in this note have been obtained by the author, with the exception of: 1. the sharp constant $\frac{1}{12}$ in Lemma \ref{lem:variance}, which was observed by ChatGPT 5.6-Sol, replacing a prior universal constant for convex $K$ by the author; 2. the counterexamples mentioned in Subsection \ref{subsec:counter}. The author's starting point in this work was Hadwiger's theorem, which implies an equivalence between the spectral and isoperimetric properties of $K_{\Lambda}$ and $\T_{\Lambda}$ under a certain tiling condition on $K_{\Lambda}$, yielding a version of Theorem \ref{thm:section}; when prompted to check, ChatGPT found that this equivalence cannot hold for general lattices. Finally, ChatGPT provided a first draft of this note, which was checked and polished by the author.

\section{A universal covariance lower bound}

We begin with the elementary Haar-measure argument behind Lemma~\ref{lem:variance}.

\begin{proof}[Proof of Lemma~\ref{lem:variance}]
Let $X$ be uniformly distributed on $K$.  Since $K$ is a fundamental domain, the random point $X+\Lambda$ is Haar-uniform on $\T_\Lambda = \R^n / \Lambda$.  Fix $0\neq\xi\in\Lambda^*$.  The character
\[
   \chi_\xi:\T_\Lambda\to\R/\Z,
   \qquad \chi_\xi(x+\Lambda)=\ip{\xi}{x}\pmod1,
\]
is a well-defined nontrivial continuous homomorphism.  It is surjective, and therefore it pushes the Haar measure on $\T_\Lambda$ forward to the Haar measure on $\R/\Z$.

Set $Y=\ip{X}{\xi}$ and $m=\mathbb EY$.  Then $Y-m\pmod1$ is uniform on $\R/\Z$.  Let $Z\in[-1/2,1/2)$ denote its centered representative.  Thus $Z$ is uniform on $[-1/2,1/2)$ and
\[
    |Y-m|\ge \dist(Y-m,\Z)=|Z|.
\]
Consequently,
\[
    \ip{\Cov_K\xi}{\xi}
        = \Var(Y)
       =\mathbb E|Y-m|^2
       \ge \mathbb E|Z|^2
       =\int_{-1/2}^{1/2}t^2\,dt
       =\frac1{12}.
\]
Choosing a shortest nonzero $\xi\in\Lambda^*$ and using
\[
   \ip{\Cov_K\xi}{\xi}
      \le \norm{\Cov_K}_{\mathrm{op}}|\xi|^2
\]
gives the spectral gap estimate.
\end{proof}

\section{The isoperimetric profile and a BV translation estimate}

We recall a standard directional translation estimate for functions of bounded variation (see \cite{MaggiBook} for missing standard definitions and results involving sets of finite perimeter and BV functions). 

\begin{lemma}[Directional $BV$ translation estimate]\label{lem:BV}
Let $A\subset\T_\Lambda$ be a set of finite perimeter, let $\partial^* A$ denote its reduced boundary, and let $z\in\R^n$.  Then
\[
   \mu_\Lambda\big(A\triangle(A+z)\big)
   \le \int_{\partial^*A}|\ip{z}{\nu_A(y)}|\,d\Per_{\T_\Lambda}(y).
\]
\end{lemma}

\begin{proof}
Write $z=te$ with $|e|=1$.  For a smooth periodic function $f$, the fundamental theorem of calculus gives
\[
    f(x+te)-f(x)=\int_0^t \partial_e f(x+se)\,ds,
\]
and hence, by translation invariance,
\[
    \int_{\T_\Lambda}|f(x+te)-f(x)|\,d\mu_\Lambda(x)
       \le |t|\int_{\T_\Lambda}|\partial_e f|\,d\mu_\Lambda.
\]
For a general periodic $BV$ function $f$, let $f_\varepsilon=f \ast \rho_\varepsilon$ denote its mollification (which remains periodic).
 Then $f_\varepsilon\to f$ in $L^1(\T_\Lambda)$ and
\[
   \partial_e f_\varepsilon=(D_e f) \ast \rho_\varepsilon,
   \qquad
   \int_{\T_\Lambda}|\partial_e f_\varepsilon|\,d\mu_\Lambda
      \le |D_e f|(\T_\Lambda) ,
\]
where $D_e f$ denotes the distributional derivative of $f$ in the direction of $e$. 
Applying the preceding estimate to $f_\varepsilon$ and passing to the
limit in $L^1$, using the translation invariance of Haar measure, yields
\[
   \int_{\T_\Lambda}|f(x+te)-f(x)|\,d\mu_\Lambda(x)
      \le |t|\,|D_e f|(\T_\Lambda).
\]
For $f=\1_A$, the structure theorem for sets of finite perimeter gives
\[
   |D_e\1_A|(\T_\Lambda)
      =\int_{\partial^*A}|\ip{e}{\nu_A}|\,d\Per_{\T_\Lambda},
\]
concluding the proof. 
\end{proof}

\begin{proof}[Proof of Theorem~\ref{thm:iso}]
Let $A\subset\T_\Lambda$ be a set of finite perimeter with $\mu_\Lambda(A)=v$.  Averaging translations over a fundamental domain gives
\begin{align*}
 \frac1{|K|}\int_K \mu_\Lambda\big(A\triangle(A+x)\big)\,dx
 =\int_{\T_\Lambda}\mu_\Lambda\big(A\triangle(A+y)\big)\,d\mu_\Lambda(y) =2v(1-v).
\end{align*}
Indeed, the last expression is the probability that two independent Haar-uniform points lie on opposite sides of $A$.

Consequently, applying Lemma~\ref{lem:BV} and Fubini's theorem, we obtain
\begin{align*}
 2v(1-v)
\le \int_{\partial^*A}
       \left(\frac1{|K|}\int_K |\ip{x}{\nu_A(y)}|\,dx\right)
       d\Per_{\T_\Lambda}(y)
\le R_1(K)\Per_{\T_\Lambda}(A).
\end{align*}
Taking infimum over all $A$ as above concludes the proof. 
\end{proof}

\section{Counterexamples to uniformly upper bounded spectral gap}

As alluded to in the Introduction, the lower bound in Corollary~\ref{cor:isotropic} cannot be complemented by a
dimension-independent upper bound, as witnessed by the following example.

\begin{example} \label{ex:counter}
 Let $m\geq 3$ be odd, put $n=2^m$, and consider the Barnes--Wall lattice $BW_m\subset\R^n$ \cite{Nebe-BarnesWall}. 
In the standard unimodular normalization, $BW_m$ is even unimodular and
\[
    \lambda_1(BW_m)^2
    =2^{(m-1)/2}=\sqrt{\frac{n}{2}}.
\]
Let $K_m$ denote the central Voronoi cell of $BW_m$. The automorphism group of
$BW_m$ contains a normal extraspecial subgroup $E_m$
whose natural representation on $\R^{2^m}$ is absolutely irreducible;
see~\cite{NebeRainsSloane-Clifford}. Since every
$g\in\textrm{Aut}(BW_m)$ is orthogonal and preserves $K_m$, we have
\[
    g\,\Cov_{K_m}\,g^T=\Cov_{K_m}.
\]
Consequently $\Cov_{K_m}$ commutes with the action of $E_m$, and
irreducibility (or Schur's lemma) yields
\[
    \Cov_{K_m}=L_m^2\, \mathrm{Id}.
\]
Thus the Voronoi cell of $BW_m$ is isotropic. Since $BW_m^*=BW_m$, it
follows that
\[
    \lambda_{SG}(\T_{BW_m})
      =4\pi^2\lambda_1(BW_m^*)^2
      =4\pi^2\sqrt{\frac{n}{2}}.
\]
In particular, even among lattices with isotropic Voronoi cells, the spectral
gap of the associated flat torus is not bounded above by a
dimension-independent constant.
\end{example}

\begin{remark} \label{rem:counter}
Other lattices provide even faster growth rate of the spectral-gap -- according to ChatGPT, there is a sequence of dimensions $n$ growing to infinity and lattices $\Lambda_n \subset \R^n$ (dual to the Craig ideal cyclotomic lattices) with $\det \Lambda_n = 1$, such that $K_{\Lambda_n}$ is isotropic and 
\[
\lambda_{1}(\Lambda_n^*) = \max_{r \in \mathbb{N}, r | n} 2 r p^{-(2r-1)/n} \simeq \frac{n}{\log(1+n)},
\]
 but we did not verify this ourselves. This is essentially the maximal growth rate, since it is known that the Hermite constant, defined as
\[
    \gamma_n:=\sup_{\Lambda \subset \R^n , \det\Lambda=1} \lambda_1(\Lambda)^2 ,
\]
satisfies
\[
    c n\leq \gamma_n\leq C n
\]
for universal constants $c,C>0$ (the upper bound follows from Minkowski's first lattice theorem \cite{GruberLekkerkerker}, and the lower bound follows from Siegel's mean-value theorem \cite{Siegel1945}).  Consequently, for every lattice $\Lambda\subset\R^n$ with $\det\Lambda=1$ (regardless of whether $K_{\Lambda}$ is isotropic or not),
\[
    \lambda_{SG}(\T_\Lambda)
    =4\pi^2\lambda_1(\Lambda^*)^2
    \leq 4\pi^2\gamma_n
    \leq C' n.
\]  
\end{remark}

\section{Spectral gap upper bound under a sectional tiling condition}

Before establishing Theorem~\ref{thm:section}, a few elementary observations are needed. For $0 \neq \xi \in \Lambda^*$, note that $\Lambda \cap \xi^{\perp}$ is of full ($n-1$) rank in $\xi^{\perp}$. Indeed, the homomorphism $\Phi_\xi : \Lambda \rightarrow \Z$ given by $\Phi_{\xi}(\lambda) = \scalar{\lambda,\xi}$ is nontrivial, and hence its image is a nonzero subgroup of $\Z$, i.e. $m \Z$ for some $m \geq 1$. Consequently $\Lambda \cap \xi^{\perp} = \ker \Phi_{\xi}$ has rank $n-1$. The dual vector $\xi$ is called primitive if $m = 1$, or equivalently, if  $t \xi \notin \Lambda^*$ for all $t \in (0,1)$.

\begin{lemma}[Covolume of a rational hyperplane]\label{lem:covol}
Let $0\neq \xi\in\Lambda^*$ be primitive and put $H=\xi^\perp$ and $\Lambda_0=\Lambda\cap H$.  Then
\[
     \det_H(\Lambda_0)=|\xi|\det(\Lambda).
\]
\end{lemma}

\begin{proof}
Since $\xi$ is primitive, there exists $\lambda_0\in\Lambda$ with $\ip{\xi}{\lambda_0}=1$.  The lattice hyperplanes parallel to $H$ are therefore separated by distance $1/|\xi|$.  A fundamental prism with base a fundamental domain of $\Lambda_0$ and height $1/|\xi|$ has volume $\det\Lambda$, and so $\det\Lambda=\det_H(\Lambda_0)/|\xi|$. 
\end{proof}

We shall also use a standard one-dimensional fact about log-concave densities: if $f$ is a log-concave probability density on $\R$, with mean zero and variance $\sigma^2$, then
\begin{equation}\label{eq:Hensley}
      \frac{1}{\sqrt{12}\,\sigma}\le f(0)\le \frac{1}{\sqrt{2}\,\sigma}.
\end{equation}
In the even case this is due to Hensley \cite{Hensley1980}, see Fradelizi \cite{Hyperplanesectionsofconvexbodiesinisotropicposition} for the general case. 

\begin{proof}[Proof of Theorem~\ref{thm:section}]
Since $K \cap H$ tiles $H$ by the translations $\Lambda \cap H$, then the latter lattice must be of full rank in $H$, and hence $H = \xi^{\perp}$ for some $0 \neq \xi \in \Lambda^*$, 
which we may assume is primitive. Up to changing its sign, we have $\xi = \nu |\xi|$. 

Now let $X$ be uniform on $K$ and let $Y=\ip{X}{\nu}$. Its density is
\[
     f_Y(t)=\frac{1}{|K|} \vol_{n-1}\big(K\cap(t \nu+H)\big) .
\]
 Since $K$ is convex, $f_Y$ is log-concave by the Brunn--Minkowski theorem \cite{AsymptoticGeometricAnalysis-Volume1}; since $b_K=0$, it has mean zero. Denote its variance by $\sigma^2:= \ip{\Cov_K \nu}{\nu}$. 
 
 By the sectional tiling hypothesis, $K\cap H$ is a fundamental domain of $\Lambda\cap H$.  
   Lemma~\ref{lem:covol} therefore gives
\[
      f_Y(0) = \frac{\vol_{n-1}(K\cap H)}{|K|}
      =\frac{\det_H(\Lambda\cap H)}{\det\Lambda}
      =|\xi|.
\]
Therefore, applying our assumption and (\ref{eq:Hensley}), we deduce
\[
\delta \norm{\Cov_K}_{op} \leq \scalar{\Cov_K \nu, \nu} = \sigma^2 \leq \frac{1}{2 f_Y(0)^2} = \frac{1}{2 |\xi|^2} \leq \frac{1}{2 \lambda_1(\Lambda^*)^2} .
\]
Together with Lemma~\ref{lem:variance}, this concludes the proof. 
\end{proof}

\end{document}